\documentclass[a4paper]{article}

\usepackage{amsmath}
\usepackage{amssymb}
\usepackage{a4}
\usepackage{amscd}

\def \dim{\mathop{\rm dim}\nolimits}

\def \cod{\mathop{\rm cod}\nolimits}
\def \corank{\mathop{\rm corank}\nolimits}

\def \M{{\mathfrak m}} 
\def \R{{\mathbb R}}
\def \C{{\mathbb C}}

\def \K{{\mathbb K}}
\def \RR{{\cal R}}
\def \AA{{\cal A}}

\def \EE{{\cal E}}

\def \KK{{\cal K}}

\def \OO{{\cal O}}

\newtheorem{theorem}{Theorem}[section]

\newtheorem{proposition}[theorem]{Proposition}

\newtheorem{remark}[theorem]{Remark}

\newtheorem{example}[theorem]{Example}

\newtheorem{definition}[theorem]{Definition}

\newenvironment{proof}
 {\begin{trivlist} \item[\hskip \labelsep {\bf Proof.}]}
 {\hfill$\Box$\end{trivlist}}

\def\X{\xi }
\def\lt{\left}
\def\rt{\right}

\begin{document}
\renewcommand{\theenumi}{(\roman{enumi})}
\normalsize

\title{$\AA $-classification of map-germs via $_V\KK $-equivalence}
\author{Kevin Houston \\
School of Mathematics \\ University of Leeds \\ Leeds, LS2 9JT, U.K. \\
e-mail: k.houston@leeds.ac.uk \\
http://www.maths.leeds.ac.uk/$\sim $khouston/\\
\ \\
Roberta Wik Atique\\
ICMC/USP – Campus de S\~{a}o Carlos\\
Caixa Postal 668\\
13560-970 S\~{a}o Carlos-SP, Brazil.
}
\date{\today }
\maketitle
\begin{abstract}
The classification of map-germs up to the natural right-left equivalence (also known as $\AA $-equivalence) is often complicated. Certainly it is more complicated than $\KK $-equivalence which is extremely easy to work with because the associated tangent spaces are not `mixed' modules as they are in the $\AA $-equivalence case.

In this paper we use a version of $\KK$-equivalence, denoted $_V\KK$-equivalence, that is defined using $\KK $-equivalences that preserve a variety in the source of maps to classify maps up to $\AA $-equivalence. This is possible through making clear the connection between the two equivalences -- previous work by Damon mostly focussed on the relation between the codimensions associated to the maps.

To demonstrate the power and efficiency of the method we give a classification of certain $\AA _e$-codimension 2 maps from $n$-space to $n+1$-space. The proof using $_V\KK $-equivalence is considerably shorter -- by a wide margin -- than one using $\AA $-equivalence directly.

58K40 58K50 MSC. Keywords: Singularities, classification.
\end{abstract}

\section{Introduction}
The classification of map-germs under $\AA $-equivalence -- a central problem in the study of singularities -- is considerably more difficult when compared to that of other equivalences such as $\RR $- and $\KK $-equivalence. One method for $\AA $-classification is to first classify up to $\KK $-equivalence and then find some method for distinguishing the $\AA $-classes within the $\KK $-class.
Similarly the classification results in \cite{cooper} and \cite{codim1} rely on the fact that for $\AA _e$-codimension one maps in the complex case the $\AA $-orbit is open in the $\KK$-orbit and hence there is only one $\AA _e$-codimension one germ.

For maps $h:(\K ^p,0)\to (\K ^q,0)$ we can consider maps under $\KK $-equivalence. If we now restrict to $\KK $-equivalences which preserve some subgerm of $(\K ^p,0)$, then we get $_V \KK $-equivalence. As one might expect this equivalence behaves much like $\KK $-equivalence. In particular algebraic structures will be similar to the $\KK $ case. These structures are in some sense simpler than those appearing when we consider  $\AA $-equivalence.

Suppose that $F:(\K ^n,S)\to (\K ^p,0)$ is a stable smooth map, $n<p$, $S$ is a finite set and $V$ is the image of $F$. For a smooth map $h:(\K ^p,0)\to (\K ^q,0)$ consider the map $h^\# (F)$ given by $h^\# (F)=F|((h\circ F)^{-1}(0),S) \to (h^{-1}(0),0)$.

Then we have the following.
\begin{theorem}
Suppose that $h:(\K^p,0)\to (\K^q,0)$ and
 $\widetilde{h}:(\K^p,0)\to (\K^q,0)$
 are submersions with $F$ transverse to $h^{-1}(0)$ and $\widetilde{h}^{-1}(0)$.
If $h^\# (F)$ and $\widetilde{h}^\# (F)$ are finitely $\AA $-determined, then
\[
h^\# (F) \sim _\AA \widetilde{h}^\# (F) \iff
h \sim_{_V \KK }  \widetilde{h}   .
\]
\end{theorem}
A more general version of this is proved in Theorem~\ref{mainthm}.
The main point of this theorem is that $\AA $-equivalence may be studied with the more amenable $_V\KK $-equivalence. In many previous classifications under $\AA $-equivalence, the calculations were extensive and formidable, often requiring computers. Here the calculations in $\AA $-equivalence tangent spaces are transferred to $_V\KK $-equivalence tangent spaces which are much simpler.
We demonstrate the simplicity and efficiency by classifying maps under $\AA $-equivalence by using a complete transversal method for $_V\KK $-equivalence. The theory for this is given in Section~\ref{sec:ct} and a classification of certain corank $\AA _e$-codimension 2 maps from $(\C ^n,0)$ to $(\C ^{n+1},0) $ is given in Section~\ref{sec:cod2classn}.
Section~\ref{sec:countereg} gives a counterexample to a plausible statement related to the long-standing Mond conjecture.

The second author thanks FAPESP for financial support: Grant \# 08/004428-7.

\section{$\AA $-equivalence and $_V\KK $-equivalence}

In the initial definitions in this section we shall mostly assume that we are working with smooth (i.e., infinitely differentiable) maps, and will note that the theory for real analytic and complex analytic maps is similar. We shall have $\K =\R $ or $\K =\C $.
We recall the definition of $\AA $-equivalence.
\begin{definition}
Two smooth map germs $f:(\K^n,S)\to(\K^p,0)$ and $\widetilde{f}:(\K^n,\widetilde{S})\to(\K^p,0)$, with $S$ and $\widetilde{S}$ finite sets, are {\bf{$\AA $-equivalent}} if there exist diffeomorphisms $\varphi$ and $\psi$ for which the following diagram commutes.
\[\begin{CD}
(\K^n,S) @>f >> (\K^p,0)\\
@V{\varphi }VV @VV{\psi}V\\
(\K^n,\widetilde{S}) @>\widetilde{f} >> (\K^p,0).
\end{CD}\]
This is also known as {\bf{Right-Left-equivalence}}.
\end{definition}
This is obviously a natural equivalence. However classifications of maps under $\AA $-equivalence have been difficult to produce. One reason is that the `tangent space' for a map is what is called a `mixed' module.

For $_V\KK$-equivalence we shall only need to work with mono-germs and so make this restriction throughout.
\begin{definition}[\cite{matherIII}]
Let $f,\widetilde{f}:(\K^p,0)\to (\K^q,0)$ be two smooth map-germs. We say that $f$ and $\widetilde{f}$ are {\bf{$\KK $-equivalent}} if there exists a diffeomorphism
$\Psi :(\K^p\times \K^q,0\times 0) \to (\K^p\times \K^q,0\times 0)$ such that
\begin{enumerate}
\item $\Psi (\K^p \times \{ 0 \} ,0\times 0) =(\K^p \times \{ 0 \} ,0\times 0)$,
\item $\Psi ( {\rm{graph}} (f) ) =  {\rm{graph}} ( \widetilde{f} ) $.
\end{enumerate}
\end{definition}
The map $\Psi $ determines a diffeomorphism $\psi: (\K^p, 0) \to (\K^p, 0)$.

It is easy to check that $\AA $-equivalent maps are $\KK $-equivalent. Though cases where the converse is true do exist (for example stable maps \cite{matherstable}, and complex $\AA _e$-codimension one maps, \cite{codim1}) the converse is not true in general.
Thus $\KK $-equivalence is weaker than $\AA $-equivalence. It is also more developed, see \cite{gibson,matherIII}. It is much easier to work with; the related tangent space for $\KK $-equivalence is an $\EE _p$-module whereas the tangent space for $\AA $-equivalence is not.


\begin{definition}
Let $(V,0)$ be subset germ of $(\K^p,0)$ (we do not assume that it is analytic). Let $h:(\K^p,0)\to (\K^q,0)$ and $\widetilde{h}:(\K^p,0)\to (\K^q,0)$ be smooth maps. We say that $h$ and $\widetilde{h}$ are {\bf{$_V\KK$-equivalent}} if
\begin{enumerate}
\item $h$ and $\widetilde{h}$ are $\KK $-equivalent, and
\item $\psi (V) =V$ for the diffeomorphism $\psi $ determined by the $\KK $-equivalence.
\end{enumerate}
\end{definition}
There are alternative ways of defining $_V\KK $-equivalence but this form will be most useful in our proofs.

Note that in the definition we do not assume that $V$ is analytic. In fact, for the classifications we have in mind $V$ is the discriminant of a stable map and hence for $\K =\R $ we can naturally have that $V$ is semi-analytic.

The concept of $_V\KK $-equivalence was introduced in \cite{rwik} to classify multi-germs from $(\C ^n,0)$ to $(\C ^{n+1},0)$ up to $\AA $-equivalence, see Theorem~2.1 of \cite{rwik}. (The notation used there is $\KK _{\cal{X}} $.) The classification in that paper was for $n=2$. The notation $_V\KK $ was later introduced by Damon in \cite{legacyiii} by analogy with his concept of $\KK _V$-equivalence which had been related to $\AA $-equivalence through a result equating $\AA _e$-codimension and $\KK _{V,e}$-codimension. See \cite{damonwark} for example.



We will show that for a large collection of maps that $\AA $-equivalence and $_V\KK $-equivalence are intimately connected -- in the right setting they are `equivalent' notions.

\begin{definition}
Suppose that $F:(\K^n,S)\to (\K^p,0)$ and $h:(\K^p,0)\to (\K^q,0)$ are smooth maps. We define the {\bf{sharp pullback}} of $F$ by $h$, denoted $h^\# (F)$, to be the multi-germ given by $F|\left( (h\circ F)^{-1}(0),S\right) \to (h^{-1}(0),0)$.
\end{definition}

%
We now state the main theorem relating $\AA $-equivalence and $_V\KK $-equivalence.
\begin{theorem}
\label{mainthm}
Suppose that $F$ and $\widetilde{F}$ are $\AA $-equivalent smooth stable maps from $(\K^n,S)$ to $(\K^p,0)$, $S$ a finite set, with discriminants $V$ and $\widetilde{V}$ respectively.
Suppose that $h$ and $\widetilde{h}$ are submersions from $(\K^p,0)$ to $(\K^q,0)$
with $F$ transverse to $h^{-1}(0)$ and
$\widetilde{F}$ transverse to $\widetilde{h}^{-1}(0)$.

If (i) either $n<p$ or $\corank (F)\neq 1 $
and
(ii) $h^\# (F)$ and $\widetilde{h}^\# (\widetilde{F})$ are finitely $\AA $-determined, then
\[
h^\# (F) \sim _\AA \widetilde{h}^\# (\widetilde{F}) \iff
h \sim_{_V \KK } \left( \widetilde{h} \circ L \right)
\]
where $L$ is the diffeomorphism arising in $\widetilde{F}= L\circ F \circ R^{-1}$.
\end{theorem}
\begin{proof}
First we restrict to the case that $\widetilde{F} = F$ and $L$ is the identity.

\noindent [$\Longrightarrow $] To lighten notation denote $h^\# (F)$ and $\widetilde{h}^\# (F)$ by $f$ and $\widetilde{f}$ respectively. Since $h^{-1}(0)$ and $\widetilde{h}^{-1}(0)$ are transverse to $F$ we can view $F$ as an unfolding of the two maps, that is, there exist commutative diagrams
\[
\begin{CD}
(\K ^n,S) @>F >> (\K ^p,0)\\
@A{i}AA @AA{j}A\\
(\K ^{n'},S' ) @>f >> (\K ^{p'},0)
\end{CD}
\hspace{0.5cm} {\text{ and }}
\hspace{0.5cm}
\begin{CD}
(\K ^n,\widetilde{S}) @>F>> (\K ^p,0)\\
@A{\widetilde{i}}AA @AA{\widetilde{j}}A\\
(\K ^{n'},\widetilde{S}') @>\widetilde{f} >> (\K ^{p'},0)
\end{CD}
\]
where $n'=n-q$, $p'=p-q$, the map $j$ is an immersion that parametrizes $h^{-1}(0)$,  $f$ is the usual pullback of $F$ by $j$ (i.e., not sharp pullback), $i$ is the natural (immersive) map arising from that pullback. There are similar definitions for $\widetilde{f}$.

As $f$ and $\widetilde{f}$ are $\AA $-equivalent there exist diffeomorphism germs $\rho :(\K^{n'},S')\to (\K^{n'},\widetilde{S}')$ and $\lambda :(\K^{p'},0)\to (\K^{p'},0)$ so that the following commutes
\[
\begin{CD}
(\K ^{n'},\widetilde{S}') @>\widetilde{f} >> (\K ^{p'},0)\\
@A{\rho }AA @AA{\lambda }A\\
(\K ^{n'},S') @>f >> (\K ^{p'},0).
\end{CD}
\]
By Theorem~3.1 on page 86 of \cite{gibetal} the unfoldings $(F,i,j)$ and $(F,\widetilde{i}\circ \rho ,\widetilde{j} \circ \lambda )$ are isomorphic as unfoldings and so there exist
diffeomorphism germs $\varphi :(\K^{n},S)\to (\K^{n},S)$ and $\psi :(\K^{p},0)\to (\K^{p},0)$ such that the following diagram commutes
\[
\begin{array}{ccccccc}
(\K^n,S) &  & & \stackrel{F}{\longrightarrow}  & & & (\K^p, 0) \\
 & \nwarrow i & & & & j \nearrow & \\
 \varphi \downarrow & & (\K^{n'} ,S') & \stackrel{f}{\longrightarrow} & (\K^{p'},0) & & \downarrow \psi \\
 & \swarrow \widetilde{i} \circ \rho  & & & &  \widetilde{j} \circ \lambda  \searrow & \\
 (\K^n,S) &  & & \stackrel{F}{\longrightarrow} & & & (\K^p, 0) \\
\end{array}
\]
As $j(\K^{p'})=h^{-1}(0)$ and $ \widetilde{j} \circ \psi (\K^{p'})=\widetilde{h}^{-1}(0)$ we have that $h$ and $\widetilde{h}$ are $\KK $-equivalent by a diffeomorphism involving $\psi $. Furthermore, as the diagram above commutes, $\psi $ has the property that it preserves the discriminant of $F$. That is, $h$ and $\widetilde{h}$ are $_V \KK $-equivalent as required.

\medskip

\noindent [$\Longleftarrow $]
First note that $h\sim _{_V\KK } \widetilde{h} $ implies that $h$ and $\widetilde{h} $ are $\KK $-equivalent and the induced diffeomorphism of $(\K^p,0)$, denoted $\psi $, preserves the discriminant of $F$, i.e., $\psi (V)=V$.
This $\KK $-equivalence implies that $h^{-1}(0)$ and $(\widetilde{h}\circ \psi )^{-1}(0)$ are diffeomorphic, i.e., that $h^{-1}(0)$ and $\widetilde{h}^{-1}(0)$ are diffeomorphic by $\psi $.
As $\psi $ preserves the discriminant of $F$ we have that $f$ and $\widetilde{f}$ have diffeomorphic discriminants.
We now use the results of \cite{dupgaffwil}. First note that as $F$ is an unfolding of $f$ we have that $f$ and $\widetilde{f}$ are not corank $1$ if $F$ is not. By assumption $f$ and $\widetilde{f}$ are finitely $\AA $-determined. Then, by Theorem~0.6(2a), Theorem~0.8(3) of \cite{dupgaffwil} and the discussion following their Theorem~0.8, the maps $f$ and $\widetilde{f}$ are right-equivalent. Hence, they are $\AA $-equivalent as required.

\medskip
\noindent[Final step] All that is left to do now is to generalize to the statement of the theorem.
From the preceding we know that
\[
h^\# (F) \sim _\AA (\widetilde{h}\circ L )^\# (F) \iff
h \sim_{_V \KK } \left( \widetilde{h} \circ L \right)   .
\]
Obviously $(\widetilde{h}\circ L)^{-1}(0)$ is mapped diffeomorphically to $\widetilde{h}^{-1}(0)$ by $L$. Also the maps $(\widetilde{h}\circ L )^\# (F) $ and $\widetilde{h}^\# (\widetilde{F}) $ have diffeomorphic discriminants by the diffeomorphism $L| (\widetilde{h}\circ L)^{-1}(0)$. Again due to the same results in \cite{dupgaffwil}
we produce an $\AA $-equivalence between
$(\widetilde{h}\circ L )^\# (F)$ and $\widetilde{h}^\# (\widetilde{F})$.
\end{proof}

%
%
%
%
\section{Complete transversals}
\label{sec:ct}

Let $\EE _p$ denote the set of germs of smooth functions on $(\K ^p,0)$. For $\K =\C$ with analytic functions this is often denoted by $\OO _p$.
The maximal ideal of $\EE _p$ will be denoted $\M _p$ or $\M $ when we can drop the $p$ without ambiguity. For a smooth map $h:(\K ^p,0)\to (\K ^q,0)$ the vector fields along $h$ are denoted by $\theta (h)$. This is isomorphic as an $\EE _p$-module to $\EE _p^q$, i.e., $q$ copies of $\EE ^p$.

For $V$ a subset germ of $(\K ^p,0)$ let $\Theta _V$ be the vector fields that integrate to give a diffeomorphism of $(\K ^p,0)$ that preserves $V$. 
Obviously $\Theta _V$ is an $\EE _p$-module with the natural structure.

We can define in the standard way a `tangent space' even though $_V\KK$ is not a Lie group.
\begin{definition}
The {\bf{$_V \KK$-tangent space of $h$ with respect to $V$}} is
\[
T_V \KK  (h) =
\{ \xi (h) \, | \, \xi \in \Theta _V \cap \M  _p\theta (h) \} + h^*(\M _q)\theta (h).
\]
\end{definition}
To produce a complete transversal result we will use diffeomorphisms with $1$-jet equal to the identity.

We let $_V\KK _1$ denote the subgroup of $_V\KK $ consisting of $\KK $-equivalences that have the diffeomorphism preserving $V$ having $1$-jet equal to the identity.

Let $\Theta _{V,1}$ be the set of all vector fields that integrate to diffeomorphisms with $1$-jet the identity and preserve $V$.

\begin{definition}
The {\bf{$_V \KK_1$-tangent space of $h$ with respect to $V$}} is
\[
T_V \KK _1  (h) =
\{ \xi (h) \, | \, \xi \in \Theta _{V,1} \} + \M _p h^*(\M _q)\theta (h).
\]
\end{definition}

\begin{theorem}[Complete transversal theorem for $_V \KK $-equivalence]
\label{thm:ct}
Suppose that $h:(\K ^p,0)\to (\K ^q,0)$ is a smooth map and $V$ is a subgerm of $(\K ^p,0)$ such that $\Theta _V$ a finitely generated $\EE _p$-module.

If $g_1, \dots ,g_s$ are homogeneous polynomial maps of degree $k+1$ such that
\[
\M ^{k+1}\theta (h)\subseteq
T_V\KK _1 (h) +{\rm{span}} \{ g_1, \dots , g_s \} + \M ^{k+2}\theta (h) ,
\]
then every $g$ with $j^{k}(h)=j^k(g)$ is $_V \KK _1$-equivalent to some $f$ where $j^{k+1}(f)$ is of the form
$j^k(h)+ \sum_{i=1}^s \alpha _i g_i $, $\alpha _i\in \K $.
\end{theorem}
\begin{proof}
The crucial parts that we require from Proposition~1.3 of \cite{bdpk} and the subsequent equation there labelled (3) can be summarized as the following:
Let $G$ be a Lie group acting smoothly on the vector space $A$ and let $W$ be a subspace of $A$
such that $l\cdot (x+w) = l\cdot x $ for all $x\in A$, $w\in W$ and $l\in TG$. ($TG$ is the tangent space.) Then,
if $x_0\in A$ and $T$ is a vector subspace of $W$ satisfying $W\subseteq T+TG\cdot x_0$, then for any $w\in W$ there exists $g\in G$, $t\in T$ such that $g\cdot (x_0+w) = x_0+t$.

Obviously $_V\KK _1$ is not a Lie group but by working with jets we have a Lie group and can apply the preceding proposition. Hence we take $G=J^{k+1}(_V\KK _1)$, the $k+1$ jet space of the $\KK$-equivalences preserving $V$; $A=J^{k+1}(p,q)$, the set of polynomial mappings of degree less than or equal to $k+1$ with zero constant term;
 $W$ is the subspace of homogeneous maps of degree $k+1$ and $T={\rm{span}} \{ g_1, \dots , g_s \}$. It is easy to show that $l\cdot (x+w) = l \cdot x $ for all $x\in A$, $w\in W$ and $l \in T_V\KK _1$. The key is that $l\in (\M_p^2)^q$, i.e., the components of $l$ are in the square of the maximal ideal. See \cite{bruce-nato} page 22.
\end{proof}
\begin{definition}
The set $\{ g_1 , g_2 ,\dots ,g_s\}$ is called a {\bf{complete transversal}} of degree $k+1$. Often we call this a $(k+1)$-complete transversal.
\end{definition}
An important part of classification is to make the complete transversal as small as possible.

\begin{definition}
The {\bf{extended $_V \KK$-tangent space of $h$ with respect to $V$}} is
\[
T_V \KK _e (h) =\{ \xi (h) \, | \, \xi \in \Theta _V \}  + h^*(\M _q)\theta (h).
\]
The {\bf{$_V \KK_e$-codimension of $h$ with respect to $V$}} is
\[
_V \KK_e -\cod (h) = \dim _\K \dfrac{\theta (h)}{T_V \KK _e(h)} .
\]
\end{definition}

We now deduce determinacy results from the complete transversal theorem. (See also Corollary~3.12 of \cite{bruce-west}.)
\begin{theorem}
\label{thm:determinacy}
Suppose that $\Theta _V$ is finitely generated as an $\EE _p$-module.
\begin{enumerate}
\item If
\[
\M ^{l+1}\theta (h)\subseteq
T_V\KK _1 (h) + \M ^{l+2}\theta (h) ,
\]
then $h$ is $l-_V\KK $-determined.
\item  If the vectors in $\Theta _V$ vanish at origin, and
\[
\M ^{l}\theta (h)\subseteq
T_V\KK _e (h) + \M ^{l+1}\theta (h) ,
\]
then $h$ is $l-_V\KK $-determined.
\end{enumerate}
\end{theorem}
\begin{proof}
(i) By Nakayama's Lemma we deduce that $\M ^{l+1}\theta (h)\subseteq
T_V\KK _1 (h) $. Then from Theorem~2.5 of \cite{bdpw} (adapted for $_V\KK $-equivalence) we deduce that $h$ is $l-_V\KK _1$-determined and hence $l-_V\KK $-determined.

(ii) We have that $T_V\KK _e(h)=T_V\KK (h)$ since the vector fields vanish at the origin. Hence by multiplication of the equation in the assumption by $\M \theta (h)$ we have
\[
\M ^{l+1}\theta (h)\subseteq \M T_V\KK  (h) + \M ^{l+2}\theta (h) .
\]
Since $\M T_V\KK (h)\subseteq T_V\KK _1 (h)$ we can apply (i) to reach the desired conclusion.
\end{proof}
Note that the corresponding theorem for $\AA $-equivalence would, for example, be that $l$-determinacy requires investigation of degree $2l+1$ functions rather than degree $l$. It is partly this difference that means use of $_V\KK $-equivalence is simpler than that of $\AA $-equivalence.

Following the usual convention in singularity theory we say that a map is $_V\KK $-finite if its $_V\KK _e$-codimension is finite.
When using $_V\KK$-equivalence in an $\AA $-equivalence classification it is useful to know the following.

\begin{proposition}
Suppose that $F:(\K ^n,0)\to (\K ^p,0)$ is a smooth stable map and $h:(\K ^p,0)\to (\K ^q,0)$ is a smooth submersion with $h$ transverse to $F$. Then $h$ is $_V\KK $-finite if and only if $h^\#(F)$ is $\AA $-finite.
\end{proposition}
\begin{proof}
This follows from Theorem~3.4 of \cite{augsmooth}. An alternative proof in the case of $\K =\C$ is to use Proposition~2.3 of \cite{damonwark} with Lemma~6.2 of \cite{legacyiii}.
\end{proof}



\section{Classification}
\label{sec:cod2classn}
In this section we shall assume for simplicity that $\K =\C$ and that we are dealing with analytic maps.
\begin{theorem}[Damon]
\label{thm:damon-main}
Suppose that $F:(\C ^n,S)\to (\C ^p,0)$ is stable map and that $h:(\C ^p,0)\to (\C ^q,0)$ is an analytic map.
Then,
\[
\AA _e - \cod (h^\# (F)) = _V\KK _e-\cod (h) .
\]
\end{theorem}
\begin{proof}
The sharp pullback of $F$ by $h$ is $\AA $-equivalent to the usual pullback of $F$ by an immersion $g$ with $h^{-1}(0)$ equal to the image of $g$. Hence by the main theorem of \cite{damonwark} we have that the $\AA _e$-codimension of $h^\# (F)$ and $\KK _{V,e}$-codimension of the pullback of $F$ by $g$ are equal. The latter is equal to $_V\KK _e-\cod (h)$ by Lemma~6.2 of \cite{legacyiii}.
\end{proof}


\begin{definition}
For $k\geq 2$ the {\bf{minimal cross cap mapping of multiplicity $k$}} is the map $\varphi _k:(\C^{2k-2},0)\to(\C^{2k-1},0)$ given by
\begin{eqnarray*}
&&\varphi _k(u_1,\ldots ,u_{k-2},v_1,\ldots ,v_{k-1},y)\\
&&\qquad\qquad\qquad\qquad=\left(u_1,\ldots ,u_{k-2},v_1,\ldots ,v_{k-1},y^k+\sum_{i=1}^{k-2}u_iy^i,\sum_{i=1}^{k-1}v_iy^i\right)
\end{eqnarray*}
We shall label the coordinates of the target $U_1,\ldots ,U_{k-2},V_1,\ldots ,V_{k-1},W_1$ and $W_2$, respectively.
The sets of coordinates will be abbreviated to $\underline{U}$, $\underline{V}$ and $\underline{W}$ respectively.
\end{definition}
\begin{remark}
Any corank $1$ stable map from $(\C ^n,0)$ to $(\C ^{n+1},0)$ is $\AA $-equivalent to the trivial unfolding of a map of the form $\varphi _k$ for some $k$. Any stable map that is equivalent to some $\varphi _k$ we shall call {\bf{minimal}}.
\end{remark}
We shall now describe the vector fields tangent to the image of $\varphi _k$. Since $\varphi _k$ is quasihomogeneous the Euler vector field
\begin{displaymath}\X_e=\lt(\begin{array}{c}
(k-1)U_1\\
(k-2)U_2\\
\vdots\\
2U_{k-2}\\
(k-1)V_1\\
(k-2)V_2\\
\vdots\\
V_{k-1}\\
kW_1\\
kW_2
\end{array}\rt)\end{displaymath}
is tangent to the image of $\varphi _k$.

In \cite{liftables} it is shown that there are three families of mappings that are tangent to the image of $\varphi _k$.
We shall denote the members of the families by $\xi ^f_j$ where $1\leq f\leq 3$ and $1\leq j\leq k-1$.
This can be written in component form as
\[
\xi ^f_j=
\left(
\begin{array}{c}
A_{1,j}^f\\
\vdots\\
A_{k-2,j}^f\\
B_{1,j}^f\\
\vdots\\
B_{k-1,j}^f\\
C_{1,j}^f\\
C_{2,j}^f
\end{array}
\right) .
\]
That is, the entries of $\xi ^f_j$ that correspond to coordinates $U_1,\dots U_{k-2}$ are labelled with $A_1, \dots , A_{k-2}$, the entries that correspond to coordinates $V_1,\dots V_{k-1}$ are labelled with $B_1,\dots , B_{k-2}$, and
the entries that correspond to coordinates $W_1$ and $W_2$ are labelled with $C_1$ and $C_2$ respectively.


\begin{theorem}[\cite{liftables}]
\label{firstfamily}

Define $U_{k-1}=V_k=0$, $U_k=1$ and $U_r=V_r=0$ for $r\le0$ and for $r>k$

There are three families of vector fields tangent to the image of $\varphi _k$. For $1\leq j \leq k-1$ the vector fields  are given by the following components.

First family:
\begin{eqnarray*}
A_{i,j}^1&=&(k-i)(k-j)U_iU_j   , \qquad 1\le i\le k-2,\\
B_{i,j}^1&=&k\sum_{r=1}^{i-1}U_{i+j-r}V_r-k\sum_{r=1}^iU_rV_{i+j-r}-(i-1)(k-j)U_jV_i\\
& & +\,kV_{i+j}W_1-kU_{i+j}W_2,  \qquad 1\le i\le k-1, \\
C_{1,j}^1&=&k(k-j)U_jW_1,\\
C_{2,j}^1&=&-kV_jW_1+(k-j)U_jW_2.
\end{eqnarray*}
Second family:
\begin{eqnarray*}
A_{i,j}^2&=&-k(k+i-j+1)U_{k+i-j+1}W_1+k\sum_{r=1}^i(k+i-j-2r+1)U_rU_{k+i-j-r+1}\\
& &-j(i+1)U_{i+1}U_{k-j}, \qquad 1\le i\le k-2,\\
B_{i,j}^2&=&-k(k+i-j+1)V_{k+i-j+1}W_1+k\sum_{r=1}^i(k+i-j-r+1)U_rV_{k+i-j-r+1}\\
& &-k\sum_{r=1}^irU_{k+i-j-r+1}V_r-j(i+1)U_{k-j}V_{i+1} , \qquad 1\le i\le k-1, \\
C_{1,j}^2&=&k(k-j+1)U_{k-j+1}W_1+jU_1U_{k-j},\\
C_{2,j}^2&=&k(k-j+1)V_{k-j+1}W_1+jV_1U_{k-j} .
\end{eqnarray*}
Third family:
\begin{eqnarray*}
A_{i,j}^3&=&-k(k+i-j+1)U_{k+i-j+1}W_2+k\sum_{r=1}^i(k+i-j-r+1)U_{k+i-j-r+1}V_r\\
&&-k\sum_{r=1}^irU_rV_{k+i-j-r+1}-k(i+1)U_{i+1}V_{k-j}, \qquad 1\le i\le k-2,\\
B_{i,j}^3&=&-k(k+i-j+1)V_{k+i-j+1}W_2+k\sum_{r=1}^i(k+i-j-2r+1)V_rV_{k+i-j-r+1}\\
&&-k(i+1)V_{i+1}V_{k-j}, \qquad 1\le i\le k-1, \\
C_{1,j}^3&=&k(k-j+1)U_{k-j+1}W_2+kU_1V_{k-j}\\
C_{2,j}^3&=&k(k-j+1)V_{k-j+1}W_2+kV_1V_{k-j}.
\end{eqnarray*}
\end{theorem}

We can now classify certain functions on the image $\varphi _k$ by considering jets.
\begin{theorem}
\label{thm:classnlem}
Suppose that $h:(\C ^{2k-1},0)\to (\C ,0)$ has finite $_V\KK _e$-codimension and that $l\geq 2$.
\begin{enumerate}
\item If $k> 2$ and $j^{l-1}(h)=U_{k-2}$, then $h\sim _{_V\KK } U_{k-2} + V_{k-1}^l$.
\item If $k> 3$ and $j^{l-1}(h)=V_{k-1}+U_{k-3}$, then $h\sim _{_V\KK } V_{k-1}+U_{k-3} + U_{k-2}^l$.
\end{enumerate}
In both cases $_V\KK _e-\cod (h)=l$ and $h$ is $l-_V\KK $-determined.
\end{theorem}
\begin{proof}
(i) We aim for an $l$-complete transversal for $j^{l-1}(h)$ and so we work modulo
$\M ^{l+1}\theta (h)$.

By inspecting the linear parts of the vector fields above, we get
from the third family all monomials of degree $l$ with $W_2, V_1, \ldots , V_{k-2}$
in $\M \,T_V\mathcal{K}j^{l-1}(h) \subseteq T_V\mathcal{K}_1j^{l-1}(h)$. From the second family we get all monomials
of degree $l$ with $W_1,U_1, \ldots ,U_{k-3}$. From $\langle j^{l-1}(h)\rangle$ we get all monomials of degree $l$
with $U_{k-2}$.

Therefore a $l$-complete transversal is $\{V^l_{k-1}\}$ and so we have $j^l(h) = U_{k-1} +
\alpha V^l_{k-1}$. Since $h$ is finitely determined we can assume that $\alpha\neq 0$. By integrating the vector field $(\xi_e-(1/k)\xi_1^{2})/k$
we obtain the following diffeomorphism that preserves $V$:
$$\varphi(\underline{U},\underline{V},\underline{W})=(U_1,\ldots,U_{k-2},e^aV_1,\ldots,e^aV_{k-1},W_1,e^aW_2).$$
So applying it we have that $j^l(h)\sim U_{k-2}+V_{k-1}^l$.

We can calculate that
$$T_V\mathcal{K}_{e}(j^l(h)) =\langle U_1, \ldots ,U_{k-2}, V_1, \ldots, V_{k-2}, V^l_{k-1},W_1,W_2\rangle.$$
(To do this calculate that $T_V\mathcal{K}_{e}(j^l(h)) + \M^{l+2}$ is the above and use Nakayama's
Lemma.) Hence $j^l(h)$ has codimension $l$. Since $\M ^l \subseteq T_V\mathcal{K}_{e}(j^l(h))$, by Theorem~3.6(ii)
the determinacy of $j^l(h)$ is $l$. Hence $h\sim_{_V\mathcal{K}} U_{k-2}+V_{k-1}^l$ and the claimed codimension
and determinacy results hold.

(ii) We work in a similar way and produce a complete transversal. From the first
family we get all terms of degree $l$ in $W_2, V_1, \ldots , V_{k-2}$ in $T_V\mathcal{K}_{1}(j^{l-1}(h))$. From the
second family we get all terms of degree $l$ in $W_1,U_1, \ldots ,U_{k-4}$. From $\xi_1^2(j^{l-1}(h))$ and
$j^{l-1}(h)$ we get all terms of degree $l$ in $U_{k-3}, V_{k-1}$. Hence, $j^l(h) = U_{k-3}+V_{k-1}+\alpha U^l_{k-2}$.

By integrating the Euler vector field, we obtain that
$$\psi(\underline{U},\underline{V},\underline{W})=(e^{(k-1)a}U_1,\ldots,e^{2a}U_{k-2},e^{(k-1)a}V_1,\ldots,e^aV_{k-1},e^{ka}W_1,e^{ka}W_2)$$
is a diffeomorphism that preserves $V$.

Now, since $\alpha\neq 0$, by change of coordinates in source and target leads to $j^l(h)\sim_{_V\mathcal{K}}\, V_{k-1}+U_{k-3}+U_{k-2}^l$.

A straightforward calculation gives that
$$T_V\mathcal{K}_{e}(j^l(h)) =\langle U_1, \ldots ,U_{k-3},U_{k-2}^{l}, V_1, \ldots, V_{k-1},W_1,W_2\rangle. $$
We can then proceed as in (i).
\end{proof}

We shall give a classification under $\AA $-equivalence. First we make a corresponding classification under $_V\KK $-equivalence. The classification in the case $k=2$ was done by \cite{bruce-west} (actually for what we might call $_V\RR $-equivalence rather than $_V\KK $-equivalence) and independently by the second author \cite{rwik}.

\begin{theorem}[Classification of codimension $2$ functions on the cross cap]
\label{thm:mainclassn}
Let $V$ be the image of the cross cap for $\K =\C$.
Suppose that $h:(\C ^{2k-1},0)\to (\C ^q,0)$ has $_V\KK _e$-codimension $2$. Then $q\leq 2$.

(a) If $q=1$, then $h$ is $_V\KK $-equivalent to one of the following $2-_V\KK $-determined functions:
\begin{enumerate}
\item $U_{k-2}+V_{k-1}^2$, $k\geq 3$,
\item $V_{k-1}+U_{k-3}+U_{k-2}^2$, $k\geq 4$,
\item $V_{2}+W_1$, $k=3$,
\item $V_{1}+W_{1}^2$, $k=2$,
\item $W_1+V_1^2$, $k=2$.
\end{enumerate}

(b) If $q=2$, then $h$ is $_V\KK $-equivalent to one of the following $2-_V\KK $-determined functions:
\begin{enumerate}
\item $(V_1,W_1)$, $k=2$,
\item $(U_1,V_2+W_1)$, $k=3$,
\item $(U_2,U_1+ V_3+W_1)$, $k=4$.
\end{enumerate}
\end{theorem}
\begin{proof}
As the vector fields tangent to $V$ all vanish at the origin we have that $e_i\notin T_V\KK _e(h)$ for all $i=1,\dots , q$, where $e_i=(0 \dots 0\ 1\ 0 \dots 0)$ is the vector with $1$ in position $i$ and zero elsewhere. Therefore $_V\KK _e-\cod (h)\geq q$. Hence $q\leq 2$.

(a) If $q=1$, then we must have $\M ^2 \subseteq T_V\KK _e(h)$, in particular by Theorem~\ref{thm:determinacy} we have that $h$ is 2-determined.

Let
\[
j^1(h) = \sum_{i=1}^{k-2} a_iU_i +  \sum_{i=1}^{k-1} b_iV_i + c_1 W_1 + c_2 W_2
\]
where $a_i, b_i, c_1, c_2 \in \C $.

Suppose that $a_{k-2}\neq 0$. We can use Mather's Lemma (Lemma~3.1 of \cite{matherstable}) to get  $j^1(h)\sim_{_V\KK } a_{k-2}U_{k-2} + b_{k-1} V_{k-1}$.

If $b_{k-1}\neq 0$, then from proof of Corollary~5.9 of \cite{liftables} $h$ would have codimension $1$.

If $b_{k-1}=0$, then again by Mather's Lemma we get that $j^1(h)=U_{k-2}$ and hence
$h\sim _{_V\KK } U_{k-2}+V_{k-1}^2$ by Theorem~\ref{thm:classnlem}.

Now suppose that $a_{k-2}=0$. If $b_{k-1}=0$, then we cannot get $U_{k-2}$ and $V_{k-1}$ in $T_V\KK (h)$ and so the codimension of $h$ would be greater than 2.

Suppose $b_{k-1}\neq 0$. If $k=3$ and $c_1\neq 0$ then by Mather's Lemma $j^1(h)\sim _{_V\KK } a_2V_2+c_1W_1$. By a change of coordinates and Theorem~\ref{thm:determinacy} we get that $h\sim _{_V\KK }V_{2}+W_1$ and has codimension 2.

Suppose $k\geq 4$. If $a_{k-3}\neq 0$, then we have $j^1(h)=a_{k-3}U_{k-3}+b_{k-1}V_{k-1}$. By a change of coordinates we get $j^1(h)\sim _{_V\KK } U_{k-3}+V_{k-1}$. From Theorem~\ref{thm:classnlem} we deduce that $h\sim _{_V\KK } U_{k-3}+V_{k-1}+U_{k-2}^2$.

If $a_{k-3}=0$, then a simple calculation shows that $1$, $U_{k-3}$ and $U_{k-2}$ are not in $T_V\KK _e(h)$ and so we get codimension at least 3.

(b) If $q=2$, then we must have $\M \,\EE _{2k-1}^2 \subseteq T_V\KK _e(h)$, and so $h$ is $1$-determined.

Let
\[
j^1(h) = (\sum_{i=1}^{k-2} a_iU_i +  \sum_{i=1}^{k-1} b_iV_i + c_1 W_1 + c_2 W_2,\sum_{i=1}^{k-2}A_iU_i +  \sum_{i=1}^{k-1} B_iV_i + C_1 W_1 + C_2 W_2)
\]
where $a_i, b_i, c_1, c_2, A_i, B_i, C_1, C_2\in \C $.

Suppose $k=2$ and $b_1C_1\neq 0$. Then after suitable change of coordinates in target we can take $b_1=1$, $c_1=0$, $C_1=1$ and $B_1=0$. Applying Mather's Lemma one can easily see that there is only one orbit, namely the one with representative $(V_1,W_1)$

Suppose $k>2$ and $a_{k-2}B_{k-1}\neq 0$. Then after suitable change of coordinates in target we can suppose $a_{k-2}=1$, $A_{k-2}=0$, $B_{k-1}=1$ and $b_{k-1}=0$.

Let $T:=T_V\KK _e(j^1(h))+{\M}^2\,\EE _{2k-1}^2$.

When $k=3$, from vector fields tangent to $V$ in the first and third family we have that $V_1e_i, W_2\,e_i\in\, T$, $i=1,2$. From  vector fields tangent to $V$ in the second family and the Euler vector field we also get $V_2\,e_2, W_1e_1\in \,T$. If $C_1\neq 0$ then all the remaining monomials of degree $1$ are in $T$ by vector fields in target. By change of coordinates in source we can take $C_1=1$. Therefore there is only one orbit with representative $(U_1,V_2+W_1)$.

When $k=4$, from vector fields tangent to $V$ in the first and third family we have that $V_j\,e_i, W_2\,e_i\in\, T$, $i,j=1,2$. From  vector fields tangent to $V$ in the second family and the Euler vector field we also get $V_3\,e_2, W_1e_1\in\, T$. If $4a_1A_1+3C_1\neq 0$ we get the remaining monomials of degree $1$ in $T$ by vector fields in source and target, since $a_1U_1e_1+C_1W_1e_2, 3U_1e_1-4A_1W_1e_2 \in T$. Therefore by  Mather's Lemma and change of coordinates in source $h\sim _{_V\KK } (U_2,U_1+ V_3+W_1)$.

Finally, if $k\geq 5$ there is no $_V\KK _e$-codimension 2 orbit. In fact, for $k=5$ and proceeding as above we get $V_je_i$, $V_4e_2$, $W_2e_i$ and $W_1e_1$ in $T$, $i=1,2$, $j=1,\ldots,3$. To obtain $U_je_i$, $V_4e_1$ and $W_1e_2$, $i=1,2$, $j=1,\ldots,3$, we have only 7 independent relations in $T$, therefore the codimension at 1-jet level is at least 1. It is not difficult to see that
$${\M}\,\EE _{9}^2\subseteq T_V\KK _e(j^1(h))+span\{W_1e_2\}+{\M}^2\,\EE _{9}^2.$$

\end{proof}


\begin{remark}
Mohammad Al-Bahadeli, a student of the first author, has extensively improved the classification. This can be found in \cite{moh-thesis}.
%
\end{remark}

Due to the nature of the maps in the above theorem we can produce an $\AA $-classification.

\begin{theorem}
Suppose that $f:(\C^n,0)\to (\C ^{n+1},0)$,
is a corank $1$ map of $\AA _e$-codimension $2$.

If there exists a $1$-parameter stable unfolding which is a minimal stable map, then $f$ is $\AA $-equivalent to a map of the form:
\begin{enumerate}
\item $(u,v,y^k+v_{k-1}^2y^{k-2} + \sum _{i=1}^{k-3} u_iy^i, \sum_{i=1}^{k-1} v_iy^i )$, $k\geq 3$,
\item $(u,v,y^k+ \sum _{i=1}^{k-2} u_iy^i,  (u_{k-3}+u_{k-2}^2)y^{k-1} + \sum_{i=1}^{k-2} v_iy^i )$, $k\geq 4$,
\item $(u_1,v_1,y^3+u_1y,y^5+v_1y)$, $k=3$,
\item $(y^2,y^3)$, $k=2$.
\end{enumerate}

If there exists a $2$-parameter stable unfolding which is a minimal stable map, then $f$ is $\AA $-equivalent to a map of the form:
\begin{enumerate}
\item $(v_1, y^3,y^5+v_1y)$, $k=3$,
\item $(u_1,v_1,v_2,y^4+u_1y,y^7+v_1y+v_2y^2+u_1y^3)$, $k=4$.
\end{enumerate}
\end{theorem}
\begin{proof}
If there exists a $1$-parameter stable unfolding, then we can assume that $f$ is the sharp pullback of a function on $(\C ^{2k-2},0)$. From Theorem~\ref{thm:damon-main} we know that $h$ has $_V\KK _e$-codimension $2$ and hence we may use the classification in Theorem~\ref{thm:mainclassn}. The mappings in the statement above arise simply from this classification.

For (iii) we have $h=V_2+W_1$ and so can replace $v_2$ by $-W_1=-(y^3+u_1y)$ as follows:
\begin{eqnarray*}
& & (u_1,v_1,y^3+u_1y, -(y^3+u_1y)y^2+v_1y)\\
&\sim _\AA &(u_1,v_1,y^3+u_1y, -y^5+(v_1+u_1^2)y)\\
&\sim _\AA &(u_1,x-u_1^2,y^3+u_1y, -y^5+xy)\\
&\sim _\AA &(u_1,x,y^3+u_1y, y^5+xy).
\end{eqnarray*}

When $k=2$ and $h=W_1+V_1^2$ one cannot apply Theorem~\ref{mainthm} since in this case  $F$ is not transverse to $h^{-1}(0)$.

If there exists a $2$-parameter stable unfolding, then we can assume that $f$ is the sharp pullback of a map $(\C ^{2k-2},0)\to (\C ^2,0)$. As above we use the classification in Theorem~\ref{thm:mainclassn}. Now change of coordinates in source and target leads to the result.

\end{proof}
\begin{remark}
The mappings in the theorem have appeared in  special case within various classifications. For example (i) corresponds to $Q_2$ and (iii) corresponds to $P_2$ in in \cite{hk}.
\end{remark}

\begin{remark}
Note that the effort involved in this classification is considerably less than a straightforward attempt to apply the complete transversal method to $\AA $-equivalence as in say \cite{hk}. It is also far simpler than the method of explicit diffeomorphisms in source and target used in \cite{cooper} to find an $\AA_e$-codimension one classification.
\end{remark}

\section{A Counterexample}
\label{sec:countereg}
The $_V\KK _e$-codimension of a map germ on a minimal cross cap is given by
\[
\dim _\C \frac{\theta (h)}{\langle \xi _e(h) , \xi_j^1 (h), \xi_j^2 (h), \xi_j^3 (h) \rangle _{j=1}^{k-1}  + h^*(\M _q)\theta (h)} .
\]
If we drop the Euler vector field, then the dimension
\[
\dim _\C \frac{\theta (h)}{\langle \xi_j^1 (h), \xi_j^2 (h), \xi_j^3 (h) \rangle _{j=1}^{k-1}  + h^*(\M _q)\theta (h)}
\]
appears, in many examples, to calculate the image Milnor number of $\varphi ^\# (h)$. In the case of $\AA _e$-codimension 1 germs on minimal cross caps it can be seen from the final remark in \cite{liftables} that the third family plays no part in the calculation of this dimension.
If, for general $_V\KK _e$-codimension, one could drop one of the families, then one would have the `right' number of generators to prove the Mond conjecture, see \cite{mond-vancyc}, for the case of corank $1$ maps from $(\C ^n,0)$ to $(\C ^{n+1},0)$. The following example is one in which all three families are required and hence is a counterexample to the suggestion that only two families are required in general. Note that this not a counterexample to the Mond conjecture.

\begin{example}
Let $V$ be the image of the multiplicity 3 minimal cross cap, denoted $\varphi _3$ above. Let $h(U_1,V_1,V_2,W_1,W_2)=(V_2+W_1,U_1)$. 
The map $h^\# (\varphi _3)$ is equivalent to $H_2$ in Mond's list \cite{mond-classn}.

We shall now calculate $\langle \xi_j^1 (h), \xi_j^2 (h), \xi_j^3 (h) \rangle _{j=1}^{k-1}  + h^*(\M _q)\theta (h)$ (this is just $T_V\KK _e(h)$ with the Euler vector field removed from the definition).
The dimension of the normal space of this conjecturally gives the image Milnor number. (This reasoning for this conjecture is not relevant and will not be given in this paper -- our objective is only to describe the normal space.) Here the normal space is generated by
$\left( \begin{array}{c} 1 \\ 0 \end{array} \right) $
and
$\left( \begin{array}{c} 0 \\ 1 \end{array} \right) $.

Applying the three families to $h$ gives
\[
\left( \begin{array}{c} W_2 \\ 0 \end{array} \right)
\left( \begin{array}{c}  V_1 \\ 0  \end{array} \right) ,
\]
\[
\left( \begin{array}{c} -2V_2+3W_1\\ 2U_1 \end{array} \right)
\left( \begin{array}{c}  V_1 \\ 3W_1  \end{array} \right),
\]
\[
\left( \begin{array}{c} W_2\\ V_1 \end{array} \right)
\left( \begin{array}{c}  0 \\ W_2  \end{array} \right) .
\]
The other vector fields generating our module are
\[
\left( \begin{array}{c} V_2+W_1 \\  0 \end{array} \right)
\left( \begin{array}{c} 0  \\ V_2+W_1 \end{array} \right)
\left( \begin{array}{c}  U_1 \\ 0 \end{array} \right)
\left( \begin{array}{c}  0\\  U_1 \end{array} \right) .
\]
Note that to the third family is required to generate the module, without it we cannot get $V_1$ and $W_2$ in the bottom row of the vector. To get $V_1$ we use the first vector of the pair
\[
\left( \begin{array}{c} W_2\\ V_1 \end{array} \right)
\]
but then we need to use the first member of the first family
\[
\left( \begin{array}{c} W_2 \\ 0 \end{array} \right)
\]
to get this.

Hence we definitely require the first and third families. However, we note that this does not give us the whole of our modified tangent space, $\langle \xi_j^1 (h), \xi_j^2 (h), \xi_j^3 (h) \rangle _{j=1}^{k-1}  + h^*(\M _q)\theta (h)$. Therefore, unlike the $\AA _e$-codimension $1$ case, we cannot hope to drop one family from our calculations to get the `right number of generators' to prove the Mond conjecture.
\end{example}

\end{document}